\documentclass[10pt]{amsart}

\makeatletter
\def\blfootnote{\gdef\@thefnmark{}\@footnotetext}
\makeatother

\usepackage{epsfig}
\usepackage{graphics}
\usepackage{dcpic, pictexwd}

\usepackage[colorlinks=true,
                    linkcolor=blue,
                    urlcolor=blue,
                    citecolor=blue,
                    anchorcolor=blue]{hyperref}
\usepackage{mathtools}
\usepackage{amsmath,amssymb}

\theoremstyle{plain}

\newtheorem*{theorem*}{Theorem}
\newtheorem*{thma}{Theorem A}
\newtheorem*{thmb}{Theorem B}

\newtheorem*{proofa}{Proof of Theorem A}
\newtheorem*{proofb}{Proof of Theorem B}
\newtheorem{theorem}{Theorem}[section]

\newtheorem{lemma}[theorem]{Lemma}

\theoremstyle{remark}
\newtheorem{remark}[theorem]{Remark}
\theoremstyle{Acknowledgments}

\theoremstyle{definition}

\def\mod{{\rm Mod}}

\begin{document}
\blfootnote{\textup{2000} \textit{Mathematics Subject Classification}:
57M07, 20F05, 20F38}
\newenvironment{prooff}{\medskip \par \noindent {\it Proof}\ }{\hfill
$\square$ \medskip \par}
    \def\sqr#1#2{{\vcenter{\hrule height.#2pt
        \hbox{\vrule width.#2pt height#1pt \kern#1pt
            \vrule width.#2pt}\hrule height.#2pt}}}
    \def\square{\mathchoice\sqr67\sqr67\sqr{2.1}6\sqr{1.5}6}
\def\pf#1{\medskip \par \noindent {\it #1.}\ }
\def\endpf{\hfill $\square$ \medskip \par}
\def\demo#1{\medskip \par \noindent {\it #1.}\ }
\def\enddemo{\medskip \par}
\def\qed{~\hfill$\square$}

 \title[On the Involution Generators of $\mod(\Sigma_{g,p})$] {On the Involution Generators of the Mapping Class Group of a Punctured Surface}

\author[T{\"{u}}l\.{i}n Altun{\"{o}}z,       Mehmetc\.{i}k Pamuk, and O\u{g}uz Y{\i}ld{\i}z ]{T{\"{u}}l\.{i}n Altun{\"{o}}z,    Mehmetc\.{i}k Pamuk, and O\u{g}uz Y{\i}ld{\i}z}

\address{Department of Mathematics, Middle East Technical University and Faculty of Engineering, Ba\c{s}kent University, Ankara, Turkey} 
\email{atulin@metu.edu.tr} 
\address{Department of Mathematics, Middle East Technical University,
 Ankara, Turkey}
 \email{mpamuk@metu.edu.tr}
 \address{Department of Mathematics, Middle East Technical University,
 Ankara, Turkey}
  \email{oguzyildiz16@gmail.com}


\begin{abstract}
Let $\mod(\Sigma_{g,p})$ denote the mapping class group of a connected orientable surface of genus $g$ with $p$ punctures. 
For every even integer $p\geq 10$ and $g\geq 14$, we prove that $\mod(\Sigma_{g, p})$ can be generated by three involutions. 
If the number of punctures $p$ is odd and $\geq 9$, we show that $\mod(\Sigma_{g, p})$ for $g\geq 13$ can be generated by four involutions. 
Moreover, we show that for an even integer $p \geq 4$ and $3\leq g \leq 6$, $\mod(\Sigma_{g,p})$ can be generated by four involutions.
\end{abstract}
\maketitle
  \setcounter{secnumdepth}{2}
 \setcounter{section}{0}
 
\section{Introduction}
Let $\Sigma_{g,p}$ denote a connected orientable surface of genus-$g$ with $p$-punctures, when $p=0$ we write  $\Sigma_{g}$. 
The \textit{mapping class group} of $\Sigma_{g,p}$ is the group of isotopy classes of orientation-preserving homeomorphisms of $\Sigma_{g,p}$ preserving the set of punctures. 

Here is a brief history of generating sets for $\mod(\Sigma_{g,p})$:
Dehn~\cite{de} showed that  $\mod(\Sigma_{g})$ can be generated by $2g(g-1)$ Dehn twists. About a quarter century later, Lickorish~\cite{l3} gave a generating set  consisting of $3g-1$ Dehn twists. 
Later, Humphries~\cite{H} reduced the number of Dehn twist generators to $2g+1$. He also proved that  the number $2g+1$ is minimal for $g\geq 2$. 
Johnson~\cite{j} proved that  the same set of Dehn twists also generates $\mod(\Sigma_{g,1})$. 
In the presence of multiple punctures, Gervais~\cite{G} proved that $\mod(\Sigma_{g,p})$ can be generated by $2g+p$ Dehn twists for $p\geq1$.

If it is not required that the generators are Dehn twists, then it is possible to obtain smaller generating sets for $\mod(\Sigma_{g,p})$:   
For $g\geq 3$ and $p=0$, Lu~\cite[Theorem~$1.3$]{lu} proved that  
$\mod(\Sigma_{g})$ can be generated by three elements.
For $g\geq 1$ and $p=0$ or $1$, 
a minimal (since the group is not abelian) generating set of two elements, a product of two Dehn twists  and a product of $2g$ Dehn twists,   
was first given by Wajnryb~\cite{w}.  Korkmaz~\cite[Theorem~$5$]{mk2} improved this result by showing that one of these two generators can be taken as a Dehn twist. 
He also showed that this group is generated by two elements of finite order ~\cite[Theorem~$14$]{mk2}.  
For $g\geq 3$, Kassabov obtained a generating set of  involution elements where the number of generators depends on $g$ and the parity of $p$  (see ~\cite[Theorem~1]{ka}).
Later, Monden~\cite{m1} removed the parity condition on $p$ for $g=7$ and $g=5$. For $g\geq1$ and $p\geq 2$, Monden~\cite{m2} also gave a generating set for $\mod(\Sigma_{g,p})$ 
consisting of three elements.  Recently, he~\cite{m3} gave a minimal generating set for $\mod(\Sigma_{g,p})$ containing two elements for $g\geq3$.

Note that any infinite group generated by two involutions must be isomorphic to the infinite dihedral group whose subgroups are either cyclic or dihedral.  Since 
$\mod(\Sigma_{g,p})$ contains nonabelian free groups, it cannot be generated by two involutions.  In this paper, we obtain the following result (cf. ~\cite[Remark~5]{ka}):

\begin{thma}\label{thma}
For every even integer $p\geq 8$ and $g\geq 14$, $\mod(\Sigma_{g,p})$ can be generated by three involutions.
Moreover, for every even integer $p\geq 4$ and for $g=3, 4, 5$ or $6$,  $\mod(\Sigma_{g,p})$ can be generated by four involutions.
\end{thma}

At the end of the paper, we also show that Theorem~A also 
holds for the cases $p=2$ or $p=3$.  For surfaces with odd number of punctures, we have the following result:

\begin{thmb}\label{thmb}
For every odd integer $p\geq 9$ and $g\geq 13$, $\mod(\Sigma_{g, p})$ is generated by four involutions.
Moreover, for every odd integer $p\geq 5$ and for $g=3, 4, 5$ or $6$,  $\mod(\Sigma_{g,p})$ can be generated by five involutions.
\end{thmb}

The paper is organized as follows: In Section~\ref{S2}, we quickly provide the necessary background on mapping class groups. 
The proofs of Theorem~A  and Theorem~B  are given in Section~\ref{S3}.

\medskip

\noindent
{\bf Acknowledgements.}
This work was supported by the Scientific and Technological Research Council of Turkey (T\"{U}B\.{I}TAK)[grant number 120F118].


\par  
\section{Background and Results on Mapping Class Groups} \label{S2}

 Let $\Sigma_{g,p}$ denote a connected orientable surface of genus $g$ with $p$ punctures specified by the set $P=\lbrace z_1,z_2,\ldots,z_p\rbrace$ of $p$ distinguished points. If $p$ is zero then we omit it from the notation and write $\Sigma_{g}$. {\textit{The mapping class group}} 
 $\mod(\Sigma_{g,p})$ of the surface $\Sigma_{g,p}$ is defined to be the group of the isotopy classes of orientation preserving
 self-diffeomorphisms of $\Sigma_{g,p}$ which fix the set $P$. {\textit{The mapping class group}} $\mod(\Sigma_{g,p})$ of the surface $\Sigma_{g,p}$ is defined to be the group of isotopy classes of all orientation preserving self-diffeomorphisms of $\Sigma_{g,p}$ which fix the set $P$. Let $\mod_{0}(\Sigma_{g,p})$ denote the subgroup of $\mod(\Sigma_{g,p})$ consisting of elements which fix the set $P$ pointwise. It is obviuos that we have the following exact sequence:
 \[
1\longrightarrow \mod_{0}(\Sigma_{g,p}) \longrightarrow \mod(\Sigma_{g,p}) \longrightarrow Sym_{p}\longrightarrow 1,
\]
where $Sym_p$ denotes the symmetric group on the set $\lbrace1,2,\ldots,p\rbrace$ and the last projection is given by the restriction of the isotopy class of a diffeomorphism to its action on the punctures. \par
Let $\beta_{i,j}$ be an embedded arc that joins two punctures $z_i$ and $z_j$ and does not intersect $\delta$ on $\Sigma_{g,p}$. 
Let $D_{i,j}$ denote a closed regular neighbourhood of $\beta_{i,j}$, which is a disk with two punctures. 
There exists a diffeomorphism $H_{i,j}: D_{i,j} \to D_{i,j}$, which interchanges the punctures such that $H_{i,j}^{2}$ is equal to the right handed Dehn twist about $\partial D_{i,j}$ and is the identity on the complement of the interior of $D_{i,j}$. Such a diffeomorphism is said to be \textit{the (right handed) half twist} about $\beta_{i,j}$. It can be extended to a diffeomorphism of $\mod(\Sigma_{g,p})$. Throughout the paper we do not distinguish a 
 diffeomorphism from its isotopy class. For the composition of two diffeomorphisms, we
use the functional notation; if $f$ and $g$ are two diffeomorphisms, then
the composition $fg$ means that $g$ is applied first and then $f$.\\
\indent
 For a simple closed 
curve $a$ on $\Sigma_{g,p}$, following ~\cite{apy,mk1}, we denote the right-handed 
Dehn twist $t_a$ about $a$ by the corresponding capital letter $A$.
Let us also remind the following basic facts of Dehn twists that we use frequently throughout  the paper: Let $a$ and $b$ be 
simple closed curves on $\Sigma_{g,p}$ and $f\in \mod(\Sigma_{g,p})$.
\begin{itemize}
\item  If $a$ and $b$ are disjoint, then $AB=BA$ (\textit{Commutativity}).
\item If $f(a)=b$, then $fAf^{-1}=B$ (\textit{Conjugation}).
\end{itemize}
Let us finish this section by noting that we  denote the conjugation relation $fgf^{-1}$ by $f^{g}$ for any $f,g \in \mod(\Sigma_{g,p})$.

\section{Involution generators for $\mod(\Sigma_{g,p})$}\label{S3}
Let us start this section by recalling the following set of generators given by Korkmaz~\cite[Theorem~$5$]{mk1}.

\begin{theorem}\label{thm1}
If $g\geq3$, then the mapping class group $\mod(\Sigma_g)$ can be generated by the four elements $R$, $A_1A_{2}^{-1}$, $B_1B_{2}^{-1}$, $C_1C_{2}^{-1}$.
\end{theorem}

Let us also recall the following well known result from algebra.
\begin{lemma}\label{lemma1}
Let $G$ and $K$ be groups. Suppose that the following short exact sequence holds,
\[
1 \longrightarrow N \overset{i}{\longrightarrow}G \overset{\pi}{\longrightarrow} K\longrightarrow 1.
\]
Then the subgroup $\Gamma$ contains $i(N)$ and has a surjection to $K$ if and only if $\Gamma=G$.
\end{lemma}
\par

In our case where $G=\mod(\Sigma_{g,p})$ and $N=\mod_{0}(\Sigma_{g,p})$, 
we have the following short exact sequence:
\[
1\longrightarrow \mod_{0}(\Sigma_{g,p})\longrightarrow \mod(\Sigma_{g,p}) \longrightarrow Sym_{p}\longrightarrow 1.
\]
Therefore, we obtain the following useful result which follows immediately from Lemma~\ref{lemma1}. Let $\Gamma$ be a subgroup of $\mod(\Sigma_{g,p})$. If the subgroup $\Gamma$ contains $\mod_{0}(\Sigma_{g,p})$ and has a surjection to $Sym_p$ then $\Gamma=\mod(\Sigma_{g,p})$.

\begin{figure}[hbt!]
\begin{center}
\scalebox{0.3}{\includegraphics{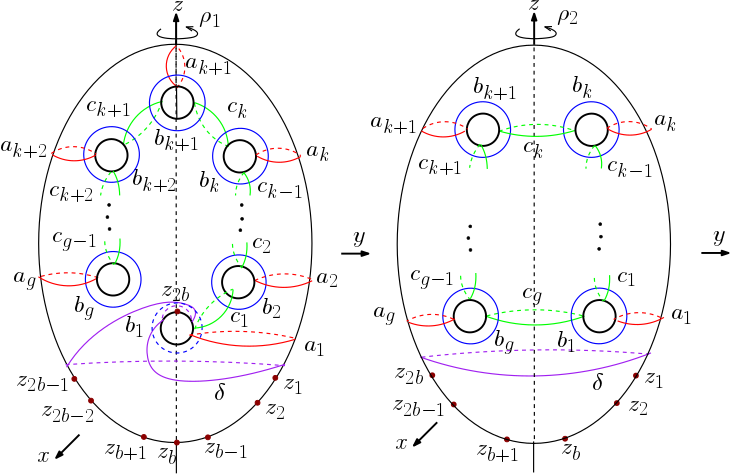}}
\caption{The involutions $\rho_1$ and $\rho_2$ if $g=2k$ and $p=2b$.}
\label{EE}
\end{center}
\end{figure}

Throughout the paper, we consider the embeddings of $\Sigma_{g,p}$ into $\mathbb{R}^{3}$ in such a way that it is invariant under the rotations 
$\rho_1$ and $\rho_2$.  Here, $\rho_1$ and $\rho_2$ are the rotations by $\pi$ about the $z$-axis  (see Figures~\ref{EE} and~\ref{OE}). 
Note that $\mod(\Sigma_{g,p})$ contains the element $R=\rho_1\rho_2$ which satisfies the following:
\begin{itemize}
\item [(i)]   $R(a_i)=a_{i+1}$, $R(b_i)=b_{i+1}$ for $i=1,\ldots,g-1$ and $R(b_g)=b_{1}$,
\item [(ii)]  $R(c_i)=c_{i+1}$ for $i=1,\ldots,g-2$,
\item [(iii)] $R(z_1)=z_p$ and $R(z_i)=z_{i-1}$ for $i=2,\ldots,p$.
\end{itemize}

We want to note here that, in the following lemmata, where we present generating sets for surfaces with even number of punctures, we mainly follow the proof 
of \cite[Theorem~$5$]{mk1}.   We use them in the proof of Theorem A and then for surfaces with odd number of punctures we explain how our arguments are modified.

\begin{lemma}\label{lemeven}
For every even integer $g=2k\geq14$ and every even integer $p=2b\geq 10$, the subgroup of $\mod(\Sigma_{g,p})$ generated by the elements 
\[
\rho_1, \rho_2 \textrm{ and }\rho_1H_{b-1,b}H_{b+1,b}^{-1}C_{k-3}B_{k-1}A_k
A_{k+2}^{-1}B_{k+3}^{-1}C_{k+4}^{-1}\]
contains the Dehn twists $A_i$, $B_i$ and $C_i$  for $i=1,\ldots,g$.
\end{lemma}

\begin{proof}

Consider the models of $\Sigma_{g,p}$ depicted in Figure~\ref{EE}. Let $F_1:=H_{b-1,b}H_{b+1,b}^{-1}C_{k-3}B_{k-1}A_k
A_{k+2}^{-1}B_{k+3}^{-1}C_{k+4}^{-1}$ and let $\Gamma$ be the subgroup of $\mod(\Sigma_{g,p})$ generated by the elements 
$\rho_1$, $\rho_2$ and $\rho_1F_1$. One can see that the elements $R=\rho_1\rho_2$ and $F_1=\rho_1 \rho_1 F_1$ are contained in
the subgroup $\Gamma$.
Let $F_2$ be the element obtained by the conjugation of $F_1$ by $R^{-3}$. Since 
\[
R^{-3}(c_{k-3}, b_{k-1}, a_k, a_{k+2}, b_{k+3}, c_{k+4}) = (c_{k-6}, b_{k-4}, a_{k-3}, a_{k-1}, b_{k}, c_{k+1})
\]
and
\[
R^{-3}(z_{b-1},z_{b},z_{b+1})=(z_{b+2},z_{b+3},z_{b+4}),
\]
\[F_2=F_{1}^{R^{-3}}=H_{b+2,b+3}H_{b+4,b+3}^{-1}C_{k-6}B_{k-4}A_{k-3}A_{k-1}^{-1}B_{k}^{-1}C_{k+1}^{-1}\in \Gamma. 
\]
\noindent 
Let $F_3$ be the element $F_1^{F_1F_2^{-1}}$, that is
$
F_3=H_{b-1,b}H_{b+1,b}^{-1}C_{k-3}A_{k-1}B_k
A_{k+2}^{-1}B_{k+3}^{-1}C_{k+4}^{-1} \in \Gamma.
$

Since we use repeatedly similar calculations in the remaining parts of the paper, let us provide  some details here. It can be shown that the diffeomorphism $F_1F_2^{-1}$ maps the curves $\lbrace c_{k-3},b_{k-1},a_k,a_{k+2},b_{k+3},c_{k+4} \rbrace$ to the curves $\lbrace c_{k-3},b_{k-1},a_k,b_{k+2},a_{k+3},c_{k+4}  \rbrace$, respectively. Also it follows from the factorizations of half twists $H_{b-1,b}H_{b+1,b}^{-1}$ and $H_{b-4,b-3}H_{b-2,b-3}^{-1}$ commute and  we get
\begin{eqnarray*}
F_3&=&F_1^{F_1F_2^{-1}}\\
&=&(F_1F_2^{-1})(H_{b-1,b}H_{b+1,b}^{-1}C_{k-3}B_{k-1}A_k
A_{k+2}^{-1}B_{k+3}^{-1}C_{k+4}^{-1})(F_1F_2^{-1})^{-1}\\
&=&H_{b-1,b}H_{b+1,b}^{-1}C_{k-3}A_{k-1}B_k
A_{k+2}^{-1}B_{k+3}^{-1}C_{k+4}^{-1}.
\end{eqnarray*}
 The subgroup $\Gamma$ contains the following elements:
\begin{eqnarray*}
F_4&=&F_{3}^{R^{-3}}=H_{b+2,b+3}H_{b+4, b+3}^{-1}C_{k-6}A_{k-4}B_{k-3}A_{k-1}^{-1}B_{k}^{-1}C_{k+1}^{-1},\\
F_5&=&F_{3}^{F_3F_4}=H_{b-1,b}H_{b+1,b}^{-1}B_{k-3}A_{k-1}B_k
A_{k+2}^{-1}B_{k+3}^{-1}C_{k+4}^{-1}.
\end{eqnarray*} 
 From these, we obtain the element 
 $F_5F_3^{-1}=B_{k-3}C_{k-3}^{-1}$, which is contained in $\Gamma$. By conjugating this element with powers of $R$, we conclude that  
 \[
 B_{i}C_{i}^{-1}\in \Gamma \  \textrm{for}  \ i=1, \ldots, g-1.
 \]
The subgroup $\Gamma$ also contains the element $F_1F_3^{-1}=B_{k-1}A_{k}B_{k}^{-1}A_{k-1}^{-1}$.
After conjugating with $R^3$ and considering the inverse, we have $A_{k+2}B_{k+3}A_{k+3}^{-1}B_{k+1}^{-1} \in \Gamma$.
This in turn  implies that for $i=1,\ldots,g-1$, the elements
\[
A_iB_{i+1}A_{i+1}^{-1}B_i^{-1}\in \Gamma.
\]
We also have the following elements in $\Gamma$:
\begin{eqnarray*}
F_6&=&F_1 (A_{k+2}B_{k+3}A_{k+3}^{-1}B_{k+2}^{-1})=H_{b-1,b}H_{b+1,b}^{-1}C_{k-3}B_{k-1}A_{k}
B_{k+2}^{-1}A_{k+3}^{-1}C_{k+4}^{-1},\\
F_7&=&F_{6}^{R^{-3}}=H_{b+2,b+3}H_{b+4,b+3}^{-1}C_{k-6}B_{k-4}A_{k-3}
B_{k-1}^{-1}A_{k}^{-1}C_{k+1}^{-1} \textrm{ and }\\
F_8&=&F_{6}^{F_6F_7}=H_{b-1,b}H_{b+1,b}^{-1}C_{k-3}B_{k-1}A_k
C_{k+1}^{-1}A_{k+3}^{-1}C_{k+4}^{-1},
\end{eqnarray*}
Hence, we can conclude that $F_8^{-1}F_6=C_{k+1}B_{k+2}^{-1}\in \Gamma$. Again conjugating with $R$  implies that
\[
C_iB_{i+1}^{-1}\in \Gamma, \ \textrm{for all} \  i=1,\ldots,g-1. 
\]
Furthermore, we can see that  $\Gamma$  contains the following elements: 
\begin{eqnarray*}
F_9&=&(B_{k-3}C_{k-3}^{-1})F_3(C_{k+4}B_{k+5}^{-1})=H_{b-1,b}H_{b+1,b}^{-1}B_{k-3}A_{k-1}B_kA_{k+2}^{-1}B_{k+3}^{-1}B_{k+5}^{-1},\\
F_{10}&=&F_{9}^{R^{-3}}=H_{b+2,b+3}H_{b+4,b+3}^{-1}B_{k-6}A_{k-4}B_{k-3}A_{k-1}^{-1}B_{k}^{-1}B_{k+2}^{-1},\\
F_{11}&=&F_{9}^{F_9F_{10}}=H_{b-1,b}H_{b+1,b}^{-1}B_{k-3}A_{k-1}B_kB_{k+2}^{-1}B_{k+3}^{-1}B_{k+5}^{-1}.
\end{eqnarray*}
From these, we obtain $F_{11}F_{9}^{-1}=B_{k+2}^{-1} A_{k+2} \in \Gamma$. By the action of $R$, we can conclude that
\[
A_iB_i^{-1}\in \Gamma \ \textrm{for all}  \ i=1, \ldots, g. 
\]
This completes the proof by  Theorem~\ref{thm1} since the subgroup $\Gamma$ contains the elements
\begin{eqnarray*}
A_1A_{2}^{-1}&=&(A_1B_1^{-1})(B_1C_1^{-1})(C_1B_{2}^{-1})(B_{2}A_{2}^{-1}),\\
B_1B_{2}^{-1}&=&(B_1C_1^{-1})(C_1B_{2}^{-1}) \textrm{ and }\\
C_1C_{2}^{-1}&=&(C_1B_{2}^{-1})(B_{2}C_{2}^{-1}).
\end{eqnarray*}
\end{proof}

\begin{figure}[hbt!]
\begin{center}
\scalebox{0.2}{\includegraphics{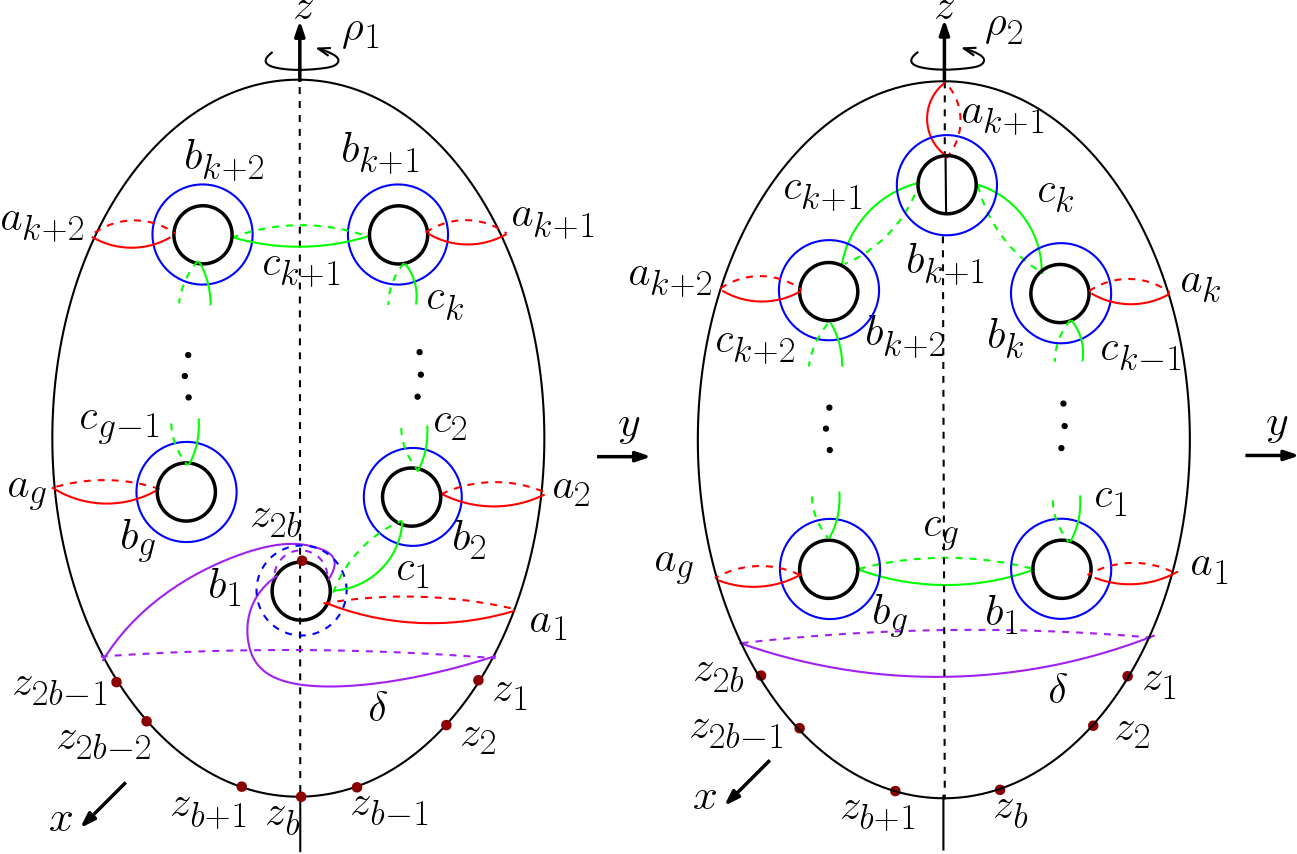}}
\caption{The involutions $\rho_1$ and $\rho_2$ for $g=2k+1$ and $p=2b$.}
\label{OE}
\end{center}
\end{figure}

If $g$ is  odd and $p$ is even, we have the following result:
\begin{lemma}\label{lemodd}
For every odd integer $g=2k+1 \geq15$ and even integer $p=2b \geq 10$, the subgroup of $\mod(\Sigma_{g,p})$  generated by the elements
\[
\rho_1,\rho_2 \textrm{ and }\rho_1H_{b-1,b}H_{b+1,b}^{-1}C_{k-3}B_kA_{k+1}A_{k+2}^{-1}B_{k+3}^{-1}C_{k+5}^{-1}
\]
 contains the Dehn twists $A_i$, $B_i$ and $C_i$  for $i=1,\ldots,g$.
\end{lemma}

\begin{proof}
Consider the models for $\Sigma_{g,p}$ as shown in Figure~\ref{OE}. Let $\Gamma$ denote the subgroup of $\mod(\Sigma_{g,p})$ 
generated by the elements $\rho_1$, $\rho_2$ and $\rho_1G_1$, where $G_1=H_{b-1,b}H_{b+1,b}^{-1}C_{k-3}B_kA_{k+1}A_{k+2}^{-1}B_{k+3}^{-1}C_{k+5}^{-1}$. 
The elements $R=\rho_1\rho_2$ and $G_1=\rho_1 \rho_1 G_1$ belong to the subgroup $\Gamma$. 
Let $G_2$ denote the conjugation of $G_1$ by $R^{-3}$,
\[
G_2=G_{1}^{R^{-3}}=H_{b+2,b+3}H_{b+4,b+3}^{-1}C_{k-6}B_{k-3}A_{k-2}A_{k-1}^{-1}B_{k}^{-1}C_{k+2}^{-1}.
\]
It is easy to verify that the element
\begin{eqnarray*}
G_3&=&G_{1}^{G_1G_2}\\
&=&H_{b-1,b}H_{b+1,b}^{-1}B_{k-3}B_kA_{k+1}A_{k+2}^{-1}C_{k+2}^{-1}C_{k+5}^{-1}
\end{eqnarray*}
is contained in $\Gamma$.
Let
\[
G_4=G_{3}^{R^{-3}}=H_{b+2,b+3}H_{b+4,b+3}^{-1}B_{k-6}B_{k-3}A_{k-2}A_{k-1}^{-1}C_{k-1}^{-1}C_{k+2}^{-1}.
\]
Thus we get the element
\begin{eqnarray*}
G_5=G_3^{G_3G_4^{-1}}
=H_{b-1,b}H_{b+1,b}^{-1}B_{k-3}C_{k-1}A_{k+1}A_{k+2}^{-1}C_{k+2}^{-1}C_{k+5}^{-1},
\end{eqnarray*}
which is contained in $G$. This implies that $G_3G_5^{-1}=B_kC_{k-1}^{-1}\in \Gamma$. By conjugating $B_kC_{k-1}^{-1}$ with powers of $R$, we see that
\[
B_{i+1}C_{i}^{-1}\in \Gamma,
\]
for all $i=1,\ldots,g-1$. In particular, the element $C_{k+5}B_{k+6}^{-1}\in \Gamma$. Hence, the subgroup $\Gamma$ contains the following element:
\[
G_6=G_1(C_{k+5}B_{k+6}^{-1})=H_{b-1,b}H_{b+1,b}^{-1}C_{k-3}B_kA_{k+1}A_{k+2}^{-1}B_{k+3}^{-1}B_{k+6}^{-1}.
\]
Then, we see that the elements
\begin{eqnarray*}
G_7&=&G_6^{R^{-3}}=H_{b+2,b+3}H_{b+4,b+3}^{-1}C_{k-6}B_{k-3}A_{k-2}A_{k-1}^{-1}B_{k}^{-1}B_{k+3}^{-1} \textrm{ and}\\
G_8&=&G_6^{G_6G_7}=H_{b-1,b}H_{b+1,b}^{-1}B_{k-3}B_kA_{k+1}A_{k+2}^{-1}B_{k+3}^{-1}B_{k+6}^{-1}
\end{eqnarray*}
are contained in $\Gamma$, which implies that the subgroup $\Gamma$ contains the element $G_6G_8^{-1}=C_{k-3}B_{k-3}^{-1}$. By the action of $R$ we see that
\[
C_iB_i^{-1} \in \Gamma
\]
for all $i=1,\ldots, g-1$.  Moreover, we get 
\begin{eqnarray*}
G_9&=&(B_{k-2}C_{k-3}^{-1})G_6=H_{b-1,b}H_{b+1,b}^{-1}B_{k-2}B_kA_{k+1}A_{k+2}^{-1}B_{k+3}^{-1}B_{k+6}^{-1} \in \Gamma,\\
G_{10}&=&G_9^{R^{-3}}=H_{b+2,b+3}H_{b+4,b+3}^{-1}B_{k-5}B_{k-3}A_{k-2}A_{k-1}^{-1}B_{k}^{-1}B_{k+3}^{-1}\in \Gamma \textrm{ and}\\
G_{11}&=&G_9^{G_9G_{10}}=H_{b-1,b}H_{b+1,b}^{-1}A_{k-2}B_kA_{k+1}A_{k+2}^{-1}B_{k+3}^{-1}B_{k+6}^{-1}\in \Gamma.
\end{eqnarray*}
From these, we have $G_9G_{11}^{-1}=B_{k-2}A_{k-2}^{-1}\in \Gamma$ so that 
\[
B_iA_i^{-1}\in \Gamma,
\]
for $i=1,\ldots,g$, by the action of $R$. The remaining part of the proof can be completed as in the proof of Lemma~\ref{lemeven}.
\end{proof}

In the following four lemmata, we give generating sets for smaller genera.
\begin{lemma}\label{lem6}
For $g=6$ and every even integer $p\geq4$, the group generated by the elements
\[
\rho_1,\rho_2, \rho_2B_2A_3A_4^{-1}B_5^{-1} \textrm{ and }\rho_1H_{b-1,b}H_{b+1,b}^{-1}C_{3}C_{4}^{-1}
\]
 contains the Dehn twists $A_i$, $B_i$ and $C_i$  for $i=1,\ldots,g$.
\end{lemma}
\begin{proof}
Consider the models for $\Sigma_{6,p}$ as shown in Figure~\ref{EE}. Let $\Gamma$ be the subgroup of $\mod(\Sigma_{6,p})$ generated by the elements $\rho_1$, $\rho_2$, $\rho_2F_1$ and $\rho_1E_1$ where $F_1=B_2A_3A_4^{-1}B_5^{-1}$ and $E_1=H_{b-1,b}H_{b+1,b}^{-1}C_{3}C_{4}^{-1}$. Hence the elements $R=\rho_1\rho_2$, $F_1=\rho_2 \rho_2 F_1$ and $E_1=\rho_1\rho_1E_1$ are contained in the subgroup $\Gamma$. 

The subgroup $\Gamma$ contains the following elements:
\begin{eqnarray*}
F_2&=&F_1^{R}=B_3A_4A_5^{-1}B_6^{-1} \in \Gamma,\\
F_{3}&=&F_{1}^{F_1F_2}=B_2B_3A_4^{-1}A_5^{-1}\in \Gamma\\
F_4&=&F_3^{R}=B_3B_4A_5^{-1}A_6^{-1}\in \Gamma
\textrm{ and}\\
F_{5}&=&F_3^{F_3F_{4}^{-1}}=B_2B_3B_4^{-1}A_5^{-1}\in \Gamma.
\end{eqnarray*}
Hence we get the element $F_5^{-1}F_3=B_4A_4^{-1}\in \Gamma$. By the action of $R$, for all $i=1,\ldots,6$,
\[
A_iB_i^{-1}\in \Gamma.
\]
Moreover, we have
\begin{eqnarray*}
F_6&=&E_1^{E_1F_3}=H_{b-1,b}H_{b+1,b}^{-1}B_{3}C_{4}^{-1} \in \Gamma
\textrm{ and}\\
F_{7}&=&E_1^{E_1F_{1}}=H_{b-1,b}H_{b+1,b}^{-1}C_{3}B_{5}^{-1}\in \Gamma.
\end{eqnarray*}
This implies that $F_6E_1^{-1}=B_3C_3^{-1}\in \Gamma$ and $F_7^{-1}E_1=B_5C_4^{-1}\in \Gamma$ and so
\[
B_iC_i^{-1}\in \Gamma \textrm{ and } B_{i+1}C_{i}^{-1}\in \Gamma,
\]
for all $i=1,\ldots,5$, by conjugating these elements with powers of $R$. The proof can be completed as in the proof of Lemma~\ref{lemeven}.
\end{proof}

\begin{lemma}\label{lem5}
For $g=5$ and every even integer $p\geq4$, the group generated by the elements
\[
\rho_1,\rho_2, \rho_1H_{b-1,b}H_{b+1,b}^{-1}A_{3}A_{4}^{-1} \textrm{ and }\rho_2A_2B_2C_2C_3^{-1}B_4^{-1}A_4^{-1} 
\]
 contains the Dehn twists $A_i$, $B_i$ and $C_i$  for $i=1,\ldots,g$.
\end{lemma}
\begin{proof}
Consider the models for $\Sigma_{5,p}$ as shown in Figure~\ref{OE}. Let $\Gamma$ denote the subgroup of $\mod(\Sigma_{5,p})$ generated by the elements $\rho_1$, $\rho_2$, $\rho_1F_1$ and $\rho_2E_1$ where $F_1=H_{b-1,b}H_{b+1,b}^{-1}A_{3}A_{4}^{-1}$ and $E_1=A_2B_2C_2C_3^{-1}B_4^{-1}A_4^{-1}$. Thus the elements $R=\rho_1\rho_2$, $F_1=\rho_1 \rho_1 F_1$ and $E_1=\rho_2\rho_2E_1$ are in the subgroup $\Gamma$. 

One can obtain the following elements:
\begin{eqnarray*}
F_2&=&F_1^{R^{-1}}=H_{b,b+1}H_{b+2,b+1}^{-1}A_{2}A_{3}^{-1} 
\\
F_{3}&=&F_2^{E_1}=H_{b,b+1}H_{b+2,b+1}^{-1}B_{2}A_{3}^{-1} \textrm{ and}\\
F_4&=&F_3^{E_1}=H_{b,b+1}H_{b+2,b+1}^{-1}C_{2}A_{3}^{-1},
\end{eqnarray*}
which are contained in $\Gamma$. Thus we get that $F_2F_3^{-1}=A_2B_2^{-1}\in \Gamma$ and $F_3F_4^{-1}=B_2C_2^{-1}\in \Gamma$. By conjugating these elements with powers of $R$, we see that 
\[
A_iB_i^{-1}\in \Gamma \textrm{ and }B_jC_j^{-1}\in \Gamma,
\]
which also implies that $A_iC_i^{-1}\in \Gamma$ for all $i=1,\ldots, 5$ and $j=1,\ldots, 4$. 
Finally, it can be verified that 
\[
E_1(a_3,c_3)=(a_3,b_4)
\]
so that the group $\Gamma$ contains the element
\[
(A_3C_3^{-1})^{E_1}=A_3B_4^{-1}.
\]
Hence $A_iB_{i+1}^{-1}\in \Gamma$ for all $i=1,\ldots,5$ by the action of $R$. The rest of the proof is similar to that of Lemma~\ref{lemeven}.
\end{proof}

\begin{lemma}\label{lem4}
For $g=4$ and every even integer $p\geq4$, the group generated by the elements
\[
\rho_1,\rho_2, \rho_2B_1A_2A_3^{-1}B_4^{-1} \textrm{ and }\rho_1H_{b-1,b} H_{b+1,b}^{-1}C_{2}C_{3}^{-1} 
\]
 contains the Dehn twists $A_i$, $B_i$ and $C_i$  for $i=1,\ldots,g$.
\end{lemma}
\begin{proof}
Let us consider the models for $\Sigma_{4,p}$ as shown in Figure~\ref{EE} and let $\Gamma$ be the subgroup of $\mod(\Sigma_{4,p})$ generated by the elements $\rho_1$, $\rho_2$, $\rho_2F_1$ and $\rho_1E_1$ where $F_1=B_1A_2A_3^{-1}B_4^{-1}$ and $E_1=H_{b-1,b} H_{b+1,b}^{-1}C_{2}C_{3}^{-1}$. Thus it is clear that the elements $R=\rho_1\rho_2$, $F_1=\rho_2 \rho_2 F_1$ and $E_1=\rho_1\rho_1E_1$ belong to the subgroup $\Gamma$.   We have the element
\[
F_2=E_1^{E_1F_1}=H_{b-1,b} H_{b+1,b}^{-1}C_{2}B_{4}^{-1}\in \Gamma.
\]
Thus the subgroup $\Gamma$ contains the elements $F{_2}^{-1} E_1=B_4C_3^{-1}$ and $\rho_1(B_4C_3^{-1})\rho_1=B_2C_2^{-1}$. By conjugating these elements with powers of $R$, we get
\[
B_{i+1}C_i^{-1}\in \Gamma \textrm{ and } B_iC_i^{-1}\in \Gamma
\]
for all $i=1,2,3$.  One can also obtain that the subgroup $\Gamma$ contains the following elements:
\begin{eqnarray*}
F_3&=&(C_1B_1^{-1})F_1=C_1A_2A_3^{-1}B_4^{-1},\\
F_4&=&F_3^{R}(B_1C_1^{-1})=C_2A_3A_4^{-1}B_1^{-1}(B_1C_1^{-1})=C_2A_3A_4^{-1}C_1^{-1}\textrm{ and}\\
F_5&=&F_3^{F_3F_4}=C_1A_2A_3^{-1}A_4^{-1}.
\end{eqnarray*}
Thus we obtain that $F_5F_3^{-1}=A_4B_4^{-1}\in \Gamma$. By the action of $R$, $A_iB_i^{-1}\in \Gamma$ for all $i=1,2,3,4$.
The remaining part of the proof is very similar to that of Lemma~\ref{lemeven}.
\end{proof}

\begin{lemma}\label{lem3}
For $g=3$ and every even $p\geq4$, the group generated by the elements
\[
\rho_1,\rho_2, \rho_1H_{b-1,b} H_{b+1,b}^{-1}A_{2}A_{3}^{-1} 
\textrm{ and }\rho_2A_1B_1C_1C_2^{-1}B_3^{-1}A_3^{-1}
\]
 contains the Dehn twists $A_i$, $B_i$ and $C_i$  for $i=1,2,3$.
\end{lemma}
\begin{proof}
Consider the models for $\Sigma_{3,p}$ as shown in Figure~\ref{OE}. Let $\Gamma$ be the subgroup of $\mod(\Sigma_{3,p})$ generated by the elements $\rho_1$, $\rho_2$, $\rho_1F_1$ and $\rho_2E_1$ where $F_1=H_{b-1,b} H_{b+1,b}^{-1}A_{2}A_{3}^{-1}$ and $E_1=A_1B_1C_1C_2^{-1}B_3^{-1}A_3^{-1}$. Thus the elements $R=\rho_1\rho_2$, $F_1=\rho_1 \rho_1 F_1$ and $E_1=\rho_2\rho_2E_1$ are contained in the subgroup $\Gamma$. We get the elements
\begin{eqnarray*}
F_2&=&F_1^{R^{-1}}=H_{b,b+1} H_{b+2,b+1}^{-1}A_{1}A_{2}^{-1}\in \Gamma,\\
F_3&=&F_2^{E_1}=H_{b,b+1} H_{b+2,b+1}^{-1}B_{1}A_{2}^{-1}\in \Gamma \textrm{ and}\\
F_4&=&F_3^{E_1}=H_{b,b+1} H_{b+2,b+1}^{-1}C_{1}A_{2}^{-1}\in \Gamma.
\end{eqnarray*}
From these, the subgroup $\Gamma$ contains the elements $F_2F_3^{-1}=A_1B_1^{-1}$ and $F_3F_4^{-1}=B_1C_1^{-1}$, which implies that $A_1C_1^{-1}\in \Gamma$. Hence
\[
A_iB_i^{-1}\in \Gamma, B_jC_j^{-1}\in \Gamma \textrm{ and } A_jC_j^{-1}\in \Gamma
\]
for all $i=1,2,3$ and $j=1,2$, by the action of $R$.  We also have the following element
\[
(A_2C_2^{-1})^{E_1}=A_2B_3^{-1},
\]
which is contained in $\Gamma$. This implies that 
\[
A_iB_{i+1}^{-1}\in \Gamma
\]
for $i=1,2$ by the action of $R$. One can complete the proof as in the proof of Lemma~\ref{lemeven}.
\end{proof}

\begin{remark}\label{podd}
Our results in lemmata~\ref{lemeven}--\ref{lem3}  are also valid for surfaces with odd number of punctures.  To see that our proofs also work 
for such surfaces we refer the reader to Figures~$3$ and $5$ in \cite{apy1}.
\end{remark}

\begin{figure}[hbt!]
\begin{center}
\scalebox{0.3}{\includegraphics{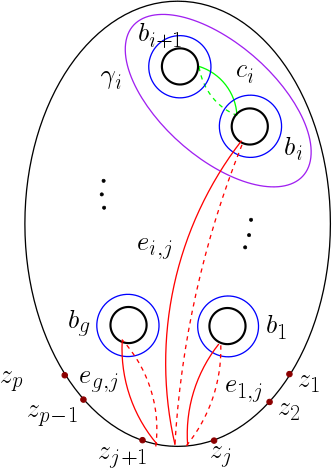}}
\caption{The curves $e_{i,j}$ and $\gamma_i$ on the surface $\Sigma_{g,p}$.}
\label{C}
\end{center}
\end{figure}

Now, in the remainder of the paper let $\Gamma$ be the subgroup of $\mod(\Sigma_{g,p})$
generated by the elements given explicitly in lemmata~\ref{lemeven}--\ref{lem3} with the conditions mentioned in these lemmata.
The proof of the following lemma is similar to that of ~ \cite[Lemma~$4.6$]{apy1}, nevertheless we give a proof for the sake of completeness of the paper.

\begin{lemma}\label{lemma4}
The group $\mod_{0}(\Sigma_{g,p})$ is contained in the group $\Gamma$.
\end{lemma}

\begin{proof}
It follows from the subgroup $\Gamma$ contains the elements $A_i$, $B_i$ and $C_j$ for all $i=1,\ldots,g$ and $j=1,\ldots,g-1$ by lemmata~\ref{lemeven}--\ref{lem3} that
it is sufficient to prove that $\Gamma$ contains the Dehn twists $E_{i.j}$ for some fixed $i$ ($j=1,2,\ldots,p-1$). Let us first note that $\Gamma$ contains $A_{g}$ and $R=\rho_1\rho_2$. Consider the models for $\Sigma_{g,p}$ as shown in Figures~\ref{EE} and~\ref{OE}. By the fact that the diffeomorphism $R$ maps $a_{g}$ to $e_{1,p-1}$, we get
\[
RA_{g}R^{-1}=E_{1,p-1} \in \Gamma.
\]
The diffeomorphism $\phi_{i}=B_{i+1}\Gamma_i^{-1}C_iB_i$ which maps each $e_{i,j}$ to $e_{i+1,j}$ for $j=1,2,\ldots,p-1$ (see Figure~\ref{C}). 
By the proof of~ \cite[Lemma~$4.5$]{apy1}, the group $\Gamma$ contains the element $\phi_{g}$. Thus we have
\[
\phi_{g-1}\cdots \phi_2\phi_1E_{1,p-1}(\phi_{g-1}\cdots \phi_2\phi_1)^{-1}=E_{g,p-1}\in H.
\]
Likewise, the diffeomorphism $R$ sends $e_{g,p-1}$ to $e_{1,p-2}$. Then we obtain
\[
RE_{g,p-1}R^{-1}=E_{1,p-2}\in \Gamma.
\]
It follows from 
\[
\phi_{g-1}\cdots \phi_2\phi_1E_{1,p-2}(\phi_{g-1}\cdots \phi_2\phi_1)^{-1}=E_{g,p-2}\in \Gamma
\]
 that
 \[
 R(E_{g,p-2})R^{-1}=E_{1,p-3}\in \Gamma
 \]
 Continuing in this way, we conclude that the elements $E_{1,1},E_{1,2},$ $\ldots,E_{1,p-1}$ belong to $\Gamma$, which completes the proof.
\end{proof}

\begin{proofa}
Consider the surface $\Sigma_{g,p}$ as in Figures~\ref{EE} and~\ref{OE}.

\underline{ If $g=2k\geq 14$ and $p=2b\geq 10$}: In this case, consider the surface $\Sigma_{g,p}$ as in Figure~\ref{EE}. Since
\[
\rho_1(c_{k-3})=c_{k+4}, \rho_1(b_{k-1})=b_{k+3} \textrm{ and }\rho_1(a_{k})=a_{k
+2},
\]
 we get
\[
\rho_1C_{k-3}\rho_1=C_{k+4},
\rho_1B_{k-1}\rho_1=B_{k+3} \textrm{ and }
\rho_1A_{k}\rho_1=A_{k+2}.
\]
Also, since  $\rho_1H_{b-1,b}\rho_1=H_{b+1,b}$,
it is easy to see that $\rho_1H_{b-1,b}H_{b+1,b}^{-1}C_{k-3}B_{k-1}A_k
A_{k+2}^{-1}B_{k+3}^{-1}C_{k+4}^{-1}$ is an involution. Therefore, the generators of the subgroup $\Gamma$ given in Lemma~\ref{lemeven} are involutions.

\underline{ If $g=2k+1\geq 13$ and $p=2b\geq10$}: In this case, consider the surface $\Sigma_{g,p}$ as in Figure~\ref{OE}. It follows from
\[
\rho_1(c_{k-3})=c_{k+5}, \rho_1(b_{k})=b_{k+3} \textrm{ and }\rho_1(a_{k+1})=a_{k
+2},
\]
 that we have
\[
\rho_1C_{k-3}\rho_1=C_{k+5},
\rho_1B_{k}\rho_1=B_{k+3} \textrm{ and }
\rho_1A_{k+1}\rho_1=A_{k+2}.
\]
Also, by the fact that $\rho_1H_{b-1,b}\rho_1=H_{b+1,b}$,
it is easy to see that the element $\rho_1H_{b-1,b}H_{b+1,b}^{-1}C_{k-3}B_kA_{k+1}A_{k+2}^{-1}B_{k+3}^{-1}C_{k+5}^{-1}$ is an involution. 
We conclude that the generators of the subgroup $\Gamma$ given in Lemma~\ref{lemodd} are involutions.

\underline{ If $g=3,4,5$ or $g=6$ and $p=2b\geq4$}:
It follows from
\begin{itemize}
\item $\rho_2(b_2)=b_5$,  $\rho_2(a_3)=a_4$ and $\rho_1(c_3)=c_4$ if $g=6$,
\item $\rho_1(a_3)=a_4$,  $\rho_2(a_2)=a_4,\rho_2(b_2)=b_4$ and $\rho_2(c_2)=c_3$ if $g=5$,
\item $\rho_2(b_1)=b_4$,  $\rho_2(a_2)=a_3$ and $\rho_1(c_2)=c_3$ if $g=4$ 
\item $\rho_1(a_2)=a_3$,  $\rho_2(a_1)=a_3,\rho_2(b_1)=b_3$ and $\rho_2(c_1)=c_2$ if $g=3$ and
\item $\rho_1H_{b-1,b}\rho_1=H_{b+1,b}$ if $g=3,4,5$ or $g=6$
\end{itemize}
that the following elements:
\begin{itemize}
\item  $\rho_2B_2A_3A_4^{-1}B_5^{-1}$ and $\rho_1H_{b-1,b}H_{b+1,b}^{-1}C_{3}C_{4}^{-1}$ if $g=6$,
\item $\rho_1H_{b-1,b}H_{b+1,b}^{-1}A_{3}A_{4}^{-1}$ and $\rho_2A_2B_2C_2C_3^{-1}B_4^{-1}A_4^{-1}$  if $g=5$,
\item $\rho_2B_1A_2A_3^{-1}B_4^{-1}$ and $\rho_1H_{b-1,b} H_{b+1,b}^{-1}C_{2}C_{3}^{-1}$ if $g=4$ and
\item $\rho_1H_{b-1,b} H_{b+1,b}^{-1}A_{2}A_{3}^{-1}$
 and $\rho_2A_1B_1C_1C_2^{-1}B_3^{-1}A_3^{-1}$ if $g=3$ 
\end{itemize}
given in lemmata~\ref{lem6}--\ref{lem3} are involutions.

Next, we show that the subgroup $\Gamma$ is equal to the mapping class group $\mod(\Sigma_{g,p})$.
By Lemma~\ref{lemma4}, the group $\mod_{0}(\Sigma_{g,p})$ is contained in the group $\Gamma$. Hence, by Lemma~\ref{lemma1}, we need to prove that 
$\Gamma$ is mapped surjectively onto $Sym_p$. The element $\rho_1\rho_2 \in \Gamma$ has the image $(1,2,\ldots,p)\in Sym_p$. 

As proven above,  the Dehn twists $A_i$, $B_i$ and $C_i$ belong to the subgroup $\Gamma$. Thus, it can be easily observed that the factorization of half twists $H_{b-1,b}H_{b+1,b}^{-1}$ are contained in subgroup $\Gamma$. Therefore, the group $\Gamma$ also contains the following element:
\[
R^{b-2}(H_{b-1,b}H_{b+1,b}^{-1})R^{2-b}=H_{1,2}H_{3,2}^{-1},
\]
which has the image $(1,2,3)\in Sym_p$. This completes the proof since the elements $(1,2,\ldots,p)$ and $(1,2,3)$ of $Sym_p$ generate the whole group $Sym_p$ if $p$ is even~\cite[Theorem B]{iz}.

\end{proofa}

When the number of punctures is odd, we introduce an additional involution $\rho_3$ (depicted in Figure~\ref{rho3})  to our generating set.  The main reason behind adding an extra involution is for generating 
the symmetric group $Sym_p$.  We want to point out that aside from generating $Sym_p$, all of our proofs in the case of even number of punctures work for odd number of punctures.
For $\rho_1$ and $\rho_2$, we distribute punctures as in Figures $3$ and $5$ in \cite{apy1} (see also Remark~\ref{podd}).
 
\begin{figure}[hbt!]
\begin{center}
\scalebox{0.35}{\includegraphics{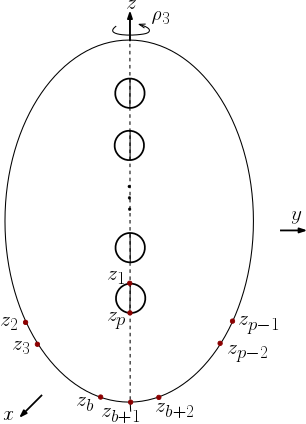}}
\caption{The involution $\rho_3$  on the surface $\Sigma_{g,p}$ for $p=2b+1$.}
\label{rho3}
\end{center}
\end{figure}

\begin{proofb}
For the first part of the proof we show that  
\begin{itemize}
\item [(i)] For every even integer $g=2k\geq14$ and every odd integer $p=2b+1\geq 9$, the subgroup $\mod_{0}(\Sigma_{g,p})$ of $\mod(\Sigma_{g,p})$ generated by the elements 
\[
\rho_1, \rho_2 \textrm{ and }\rho_1H_{b-1,b}H_{b+1,b}^{-1}C_{k-3}B_{k-1}A_k
A_{k+2}^{-1}B_{k+3}^{-1}C_{k+4}^{-1}, \rho_3. 
\]

\item[(ii)] For every odd integer $g=2k+1 \geq15$ and odd integer $p=2b+1 \geq 9$, the subgroup $\mod_{0}(\Sigma_{g,p})$ of $\mod(\Sigma_{g,p})$  generated by the elements
\[
\rho_1,\rho_2 \textrm{ and }\rho_1H_{b-1,b}H_{b+1,b}^{-1}C_{k-3}B_kA_{k+1}A_{k+2}^{-1}B_{k+3}^{-1}C_{k+5}^{-1},  \rho_3.
\]
\end{itemize}
Note that, it is enough to prove that the subgroup generated by the elements above mapped surjectively onto $Sym_p$.  
For this, consider the images of the elements $\rho_1$, $\rho_2$ and $\rho_3$ 
\begin{align*}
&(1, p-1) (2, p-2) \ldots (b, b+1),\\
&(1,p) (2, p-1) \ldots (b, b+2), \\
&(2, p-1) (3, p-2) \ldots (b, b+2).
\end{align*}
This finishes the proof for the first part,since these elements generate $Sym_p$, see \cite[Lemma 6]{m1}.  For the second part of the theorem, note that adding $\rho_3$ to the corresponding 
generating set given in Theorem A, finishes the proof.
\end{proofb}

As a last observation, one can prove that Theorem~A also holds  for the cases $p=2$ or $p=3$.
In theses cases, the generating set of $\Gamma$ can be chosen as 
\[
\begin{array}{lll}
\lbrace \rho_1,\rho_2,\rho_1C_{k-3}B_{k-1}A_k
A_{k+2}^{-1}B_{k+3}^{-1}C_{k+4}^{-1} \rbrace & \textrm{if} & g=2k\geq14,\\
\lbrace \rho_1,\rho_2,\rho_1C_{k-3}B_kA_{k+1}A_{k+2}^{-1}B_{k+3}^{-1}C_{k+5}^{-1} \rbrace & \textrm{if} & g=2k+1\geq13.\\
\lbrace \rho_1,\rho_2,\rho_2B_2A_3A_4^{-1}B_5^{-1},\rho_1C_{3}C_{4}^{-1} \rbrace & \textrm{if} & g=6.\\
\lbrace \rho_1,\rho_2,\rho_1A_{3}A_{4}^{-1},\rho_2A_2B_2C_2C_3^{-1}B_4^{-1}A_4^{-1} \rbrace & \textrm{if} & g=5.\\
\lbrace \rho_1,\rho_2,\rho_2B_1A_2A_3^{-1}B_4^{-1},\rho_1C_{2}C_{3}^{-1} \rbrace & \textrm{if} & g=4.\\
\lbrace \rho_1,\rho_2,\rho_1A_{2}A_{3}^{-1},\rho_2A_1B_1C_1C_2^{-1}B_3^{-1}A_3^{-1} \rbrace & \textrm{if} & g=3.\\
\end{array}.
\]
It can be easily proven that the group $\Gamma$ contains $\mod_{0}(\Sigma_{g,p})$ by the similar arguments in the proofs of  lemmata~\ref{lemeven}--\ref{lem3}. 
The element $\rho_1\rho_2 \in \Gamma$ has the image $(1,2,\ldots,p)\in Sym_p$. Hence, this element generates $Sym_p$ for $p=2$. If $p=3$, we distribute the punctures 
as in ~ \cite[Figure~$1$]{ka}. Then the element $\rho_1$ has the image $(1,3)$, which generate $Sym_p$ together with the element $(1,2,3)$. Therefore, the group $\Gamma$ 
is mapped surjectively onto $Sym_p$ for $p=2,3$. One can conclude that the group $\Gamma$ is equal to $\mod(\Sigma_{g,p})$.


\end{document}